\documentclass[12pt]{amsart}
\usepackage{amssymb,amsthm,amsmath,epsfig,latexsym}
\usepackage{calc}
\usepackage{times}
\begin{document}

\newcommand{\mmbox}[1]{\mbox{${#1}$}}
\newcommand{\proj}[1]{\mmbox{{\mathbb P}^{#1}}}
\newcommand{\Cr}{C^r(\Delta)}
\newcommand{\CR}{C^r(\hat\Delta)}
\newcommand{\affine}[1]{\mmbox{{\mathbb A}^{#1}}}
\newcommand{\Ann}[1]{\mmbox{{\rm Ann}({#1})}}
\newcommand{\caps}[3]{\mmbox{{#1}_{#2} \cap \ldots \cap {#1}_{#3}}}
\newcommand{\N}{{\mathbb N}}
\newcommand{\Z}{{\mathbb Z}}
\newcommand{\R}{{\mathbb R}}
\newcommand{\Tor}{\mathop{\rm Tor}\nolimits}
\newcommand{\Ext}{\mathop{\rm Ext}\nolimits}
\newcommand{\Hom}{\mathop{\rm Hom}\nolimits}
\newcommand{\im}{\mathop{\rm Im}\nolimits}
\newcommand{\rank}{\mathop{\rm rank}\nolimits}
\newcommand{\supp}{\mathop{\rm supp}\nolimits}
\newcommand{\arrow}[1]{\stackrel{#1}{\longrightarrow}}
\newcommand{\CB}{Cayley-Bacharach}
\newcommand{\coker}{\mathop{\rm coker}\nolimits}
\sloppy
\newtheorem{defn0}{Definition}[section]
\newtheorem{prop0}[defn0]{Proposition}
\newtheorem{conj0}[defn0]{Conjecture}
\newtheorem{thm0}[defn0]{Theorem}
\newtheorem{lem0}[defn0]{Lemma}
\newtheorem{corollary0}[defn0]{Corollary}
\newtheorem{example0}[defn0]{Example}
\newtheorem{remark0}[defn0]{Remark}

\newenvironment{defn}{\begin{defn0}}{\end{defn0}}
\newenvironment{prop}{\begin{prop0}}{\end{prop0}}
\newenvironment{conj}{\begin{conj0}}{\end{conj0}}
\newenvironment{thm}{\begin{thm0}}{\end{thm0}}
\newenvironment{lem}{\begin{lem0}}{\end{lem0}}
\newenvironment{cor}{\begin{corollary0}}{\end{corollary0}}
\newenvironment{exm}{\begin{example0}\rm}{\end{example0}}
\newenvironment{rem}{\begin{remark0}\rm}{\end{remark0}}

\newcommand{\defref}[1]{Definition~\ref{#1}}
\newcommand{\propref}[1]{Proposition~\ref{#1}}
\newcommand{\thmref}[1]{Theorem~\ref{#1}}
\newcommand{\lemref}[1]{Lemma~\ref{#1}}
\newcommand{\corref}[1]{Corollary~\ref{#1}}
\newcommand{\exref}[1]{Example~\ref{#1}}
\newcommand{\secref}[1]{Section~\ref{#1}}
\newcommand{\remref}[1]{Remark~\ref{#1}}

\newcommand{\std}{Gr\"{o}bner}
\newcommand{\jq}{J_{Q}}



\title{On freeness of divisors in $\mathbb P^2$}

\author{\c Stefan O. Toh\v aneanu}
\address{Department of Mathematics\\ The University of Western Ontario\\ London, Ontario N6A 5B7\\}
\email{stohanea@uwo.ca}


\begin{abstract}
\noindent Let $I\subset \mathbb C[x,y,z]$ be an ideal of height 2 and minimally generated by three homogeneous polynomials of the same degree. If $I$ is a locally complete intersection we give a criterion for $\mathbb C[x,y,z]/I$ to be arithmetically Cohen-Macaulay. Since the setup above is most commonly used when $I=J_F$ is the Jacobian ideal of the defining polynomial of a ``quasihomogeneous'' reduced curve $Y=V(F)$ in $\mathbb P^2$, our main result becomes a criterion for freeness of such divisors. As an application we give an upper bound for the degree of the reduced Jacobian scheme when $Y$ is a free rank 3 central essential arrangement, as well as we investigate the connections between the first syzygies on $J_F$, and the generators of $\sqrt{J_F}$.
\end{abstract}
\maketitle

\section{Introduction}

In the landmark paper of Saito (\cite{s}) it was introduced the concept of a free divisor $Y$ on a smooth algebraic variety $X$. A divisor $Y$ on $X$ is {\em free} if the $\mathcal O_X-$module $$Der_X(-\log Y):=\{\theta\in Der|\theta(\mathcal O_X(-Y))\subseteq \mathcal O_X(-Y)\},$$ is free, where $\mathcal O_X$ is the sheaf of regular functions on $X$.

Terao specialized the study to the case when $Y$ is an arrangement of hyperplanes in a vector space, obtaining amazing results in both the algebraic and topological directions of study of such divisors. We mention just a few of these results: Terao's Factorization Theorem (\cite{t}) shows that the Poincar\'{e} polynomial of the complement of a free arrangement factors completely, and Terao's Addition-Deletion Theorem (\cite{t2}) which relates the freeness of $Y$ to the freeness of $Y'$, the hyperplane arrangement obtained from $Y$ by removing a hyperplane. Also, Terao conjectured that over a field of characteristic 0, the condition to be free depends only on the intersection lattice (for background related to hyperplane arrangements one should check \cite{ot}).

This early success, and the very difficult conjecture mentioned above, determined a whole pleiad of mathematicians to look for various criteria of freeness for hyperplane arrangements, and more generally, for other type of divisors. An interesting nice criterion is due to Yoshinaga (\cite{y1}, Corollary 3.3), and relates the freeness of an arrangement of lines in $\mathbb P^2$ to the restricted multiarrangement. This criterion was later generalized by Schulze (\cite{sc}, Theorem 2) to higher rank hyperplane arrangements.

Outside the world of hyperplane arrangements, the tests for freeness are specialized to divisors with quasihomogeneous singularities. Without this condition, as we can see in \cite{st}, Example 4.1 and Example 4.2, Terao's Conjecture is not true anymore (in the sense that there exists two arrangements of lines and conics with the same real picture, but one free and the other not free), and even Yoshinaga's criterion fails to work. For the case of quasihomogeneous line-conic arrangements, the same paper presents a similar method to Terao's Addition-Deletion Theorem. At the other end, i.e. when the divisor is irreducible, it seems that freeness and quasihomogeneity of the singularities should also be put together (see \cite{sim3}, Proposition 4.4). Note that for divisors in $\mathbb P^2$, quasihomogeneity and linear type property are equivalent notions. In affine space, or in more dimensions, this is not true anymore. Also, \cite{dssww} goes to the extent of studying freeness of general divisors $Y$ on any $X$, when $Y$ admits locally an Euler vector field (i.e., ``quasihomogeneous'').

Our notes follow the same path; we give a criterion for freeness for quasihomogeneous curves in $\mathbb P^2$. In a more general setting we prove the following:

\vskip .1in

\noindent \textbf{Main Result:} Let $I\subset R=\mathbb C[x,y,z]$ be a height $2$ locally a complete intersection ideal, minimally generated by three homogeneous polynomials of the same degree. Then, $R/I$ is arithmetically Cohen-Macaulay if and only if there exists a syzygy on $I$ forming an $R-$regular sequence.

\vskip .1in

This result is in the spirit of a note of Eisenbud and Huneke (\cite{eh}); the title says it all: ``Ideals with a regular sequence as syzygy''. Their theorem is the following:

\begin{thm} (\cite{eh}, Theorem)\label{thm:Thm0} Let $R$ be a local Noetherian ring, let $f_1,\ldots,f_k\in R$ such that $$s_1f_1+\cdots+s_kf_k=0,$$ where $s_1,\ldots,s_k\in R$ is a regular sequence. Let $I=\langle f_1,\ldots,f_k\rangle$ and suppose that $grade(I)=k-1,$ the maximal possible value. Then
\begin{enumerate}
  \item if $k$ is odd, then $R/I$ is perfect of Cohen-Macaulay type 2.
  \item if $k$ is even, there exists an element $f\notin I$ such that $I:\langle s_1,\ldots,s_n\rangle=\langle I,f\rangle$, and $\langle I,f\rangle$ is perfect of Cohen-Macaulay type 1.
\end{enumerate}
\end{thm}

It is very important to mention that Kustin, in \cite{k}, gives a complete answers in regard to the graded minimal free resolutions of ideals $I$ and $\langle I,f\rangle$, where $I$ is generated by the entries in the product of a vector of variables with an alternating square matrix of some other variables, and $f$ is the Pfaffian of this square matrix. The specialization from this generic case to the ideals in the above theorem is done in \cite{eh}.

\vskip .1in

Our case of interest is the situation when $I=J_F=\langle \frac{\partial F}{\partial x_0},\ldots,\frac{\partial F}{\partial x_{k-1}}\rangle \subset\mathbb C[x_0,\ldots,x_{k-1}]$ is the Jacobian ideal of a divisor $Y=V(F)\subset\mathbb P^{k-1}, k\geq 3,$ with $ht(J_F)=2$ (e.g., $Y$ is an arrangement of hyperplanes). In this case $grade(J_F)=2$ and it can be maximal in the sense of the theorem above only when $k=3$. Despite the beauty of Theorem \ref{thm:Thm0}, we can use only the first part with $k=3$. But in this particular instance, from the definition of regular sequences, the result is a restatement of the Hilbert-Burch Theorem (\cite{e}, Theorem 20.15). The restriction to divisors on $\mathbb P^2$ is made not only because of the above considerations, but also due to an example (Example \ref{exm:P3}) when the natural generalization of this criterion to higher dimensions does not work.

We should mention that the local complete intersection condition in our result means exactly that the ideal $I$ is of linear type (see \cite{AA}, Lemma 3.1 and Corollary 3.2). So, under the linear type assumption, with Corollary 3.12 in \cite{sim1}, we can add an equivalent statement to the nice criterion for freeness of divisors in $\mathbb P^2$ obtained in \cite{sim1}, Proposition 4.1.

In Section 2 we give a detailed proof of our result. One implication is immediate from Theorem \ref{thm:Thm0}, and the other implication uses several steps. In Section 3 we give an application of this criterion to obtain an upper bound for the degree of the reduced Jacobian scheme of a free line arrangements in $\mathbb P^2$. We also study the first homological properties of this zero-dimensional reduced scheme, in connection to the syzygies on the Jacobian ideal of the divisor.

\section{Almost complete intersections of height 2 in $\mathbb C[x,y,z]$}

Let $I$ be an ideal in $R=\mathbb C[x,y,z]$, the ring of homogeneous polynomials with coefficients in $\mathbb C$, the field of complex numbers. Also, we assume that the height (or codimension) of $I$ is $ht(I)=2$, and $I$ is minimally generated by three homogeneous polynomials of the same degree.

By \cite{e}, $R/I$ is arithmetically Cohen-Macaulay if and only if $R/I$ has the Hilbert-Burch minimal free resolution $$0\rightarrow R^2\rightarrow R^3\rightarrow R\rightarrow R/I\rightarrow 0.$$

A first syzygy on $I$ is a $3-$tuple $(A,B,C), A,B,C\in R$, such that $$Af_1+Bf_2+Cf_3=0,$$ for some minimal generating set of homogeneous polynomials $\{f_1,f_2,f_3\}$ for the ideal $I$ (i.e., $I=\langle f_1,f_2,f_3\rangle$). A syzygy $(A,B,C)$ on $I$ will be called \textit{regular} if $A,B,C$ form an $R-$regular sequence.

$I$ is locally a complete intersection iff $I_{\underline{p}}\subset R_{\underline{p}}$ is generated by two elements, for all minimal prime ideals $\underline{p}$ of $I$. By \cite{cs}, Theorem 1.7, $I$ is locally a complete intersection if and only if the only syzygies $(A,B,C)$ on $I$ with $A,B,C\in I^{sat}$ are the Koszul syzygies. This result of Cox and Schenck was the initial inspirational point for our main result: observe that the syzygies they considered have $V(A,B,C)\neq \emptyset$, and our question addresses when does $I$ has a syzygy $(A,B,C)$, with $V(A,B,C)=\emptyset$?

In what follows we will use extensively \cite{e}, Corollary 17.7. In our setup ($R$ is a local ring of maximal ideal $\langle x,y,z\rangle$) this result translates to the following: a set of three homogeneous polynomials $\{A,B,C\}$ in $R$ forms a regular sequence if and only if $ht(A,B,C)=3$ (equivalently $V(A,B,C)=\emptyset$ as a set in $\mathbb P^2$).

\begin{thm}\label{thm:main} Let $I\subset R$ be a height 2 ideal locally a complete intersection, minimally generated by three polynomials of the same degree. Then $R/I$ is arithmetically Cohen-Macaulay if and only if there exists a regular syzygy on $I$.
\end{thm}
\begin{proof} $``\Leftarrow''$ Let $(A,B,C)$ be a regular syzygy on $I$. This means that there exist $f_1,f_2,f_3\in R$ such that $I=\langle f_1,f_2,f_3\rangle$ and $$Af_1+Bf_2+Cf_3=0$$ with $\{A,B,C\}$ forming a regular sequence in $R$.

From the syzygy equation above and from the definition of a regular sequence, we have $f_1\in\langle B,C\rangle$, and so $f_1=BD+CE$. Therefore $$B(AD+f_2)+C(AE+f_3)=0,$$ which will give that $AD+f_2=CF$ and $AE+f_3=-BF$. We get that $f_1,-f_2,f_3$ are the $2\times2$ minors of the matrix $$\left[
\begin{array}{cc}
A&F \\
B&-E \\
C&D
\end{array}
\right].$$ But this means that $R/I$ has the desired Hilbert-Burch minimal free resolution (\cite{e}, Theorem 20.15).

\vskip .2in

$``\Rightarrow''$\footnote{The referee suggested an alternative shorter proof of this implication using Prime Avoidance (\cite{e}, Lemma 3.3), yet the details of this approach remains to be clarified.} Suppose $R/I$ has a minimal free resolution $$0\rightarrow R^2\stackrel{\phi}\rightarrow R^3\rightarrow R\rightarrow R/I\rightarrow 0,$$ where $\phi=\left[
\begin{array}{cc}
A_1&A_2 \\
B_1&B_2 \\
C_1&C_2
\end{array}
\right]$. Let $$f_1=B_1C_2-B_2C_1, f_2=A_1C_2-A_2C_1, f_3=A_1B_2-A_2B_1,$$ be the $2\times 2$ minors of $\phi$ which will minimally generate the ideal $I$.

We are going to show that there exist a syzygy $(A,B,C)$ on these generators (i.e., $Af_1+Bf_2+Cf_3=0$) with $\{A,B,C\}$ forming a regular sequence.

On a side note, every minimally generating set of $I$ will have a regular syzygy: if $I=\langle g_1,g_2,g_3\rangle$, then $$(g_1,g_2,g_3)= (f_1,f_2,f_3)\mathcal M,$$ for some invertible $3\times 3$ matrix $\mathcal M$ with entries in $\mathbb C$. If $\{A,B,C\}$ is a regular sequence with $Af_1+Bf_2+Cf_3=0,$ then $\{A',B',C'\}$, where $\left[
\begin{array}{c}
A' \\
B' \\
C'
\end{array}
\right]=\mathcal M^{-1}\cdot \left[
\begin{array}{cc}
A \\
B \\
C
\end{array}
\right]$, is also a regular sequence with $A'g_1+B'g_2+C'g_3=0$.

\vskip .2in

If $\{A_1,B_1,C_1\}$ is a regular sequence or $\{A_2,B_2,C_2\}$ is a regular sequence, we are done. Suppose neither of them is, and suppose that $\deg(A_1)=\deg(B_1)=\deg(C_1)=d_1$ and $\deg(A_2)=\deg(B_2)=\deg(C_2)=d_2$, with $d_2\geq d_1$. The goal is to show that there exists $f\in R_{d_2-d_1}$ such that $\{A_2-fA_1,B_2-fB_1,C_2-fC_1\}$ is a regular sequence.

\vskip .1in

\noindent\textit{Step 1.} First we show that $Z=V(A_1,A_2,B_1,B_2,C_1,C_2)=\emptyset$. If $[a,b,c]\in Z$ is a point, then $[a,b,c]\in V(I)$. Suppose $\underline{p}$ is the ideal of the point $[a,b,c]$. Localizing at $\underline{p}$, and since $I$ is locally a complete intersection, we obtain $$0\rightarrow R_{\underline{p}}^2 \rightarrow R_{\underline{p}}^3\rightarrow I_{\underline{p}}\rightarrow 0$$ and $$0\rightarrow R_{\underline{p}} \rightarrow R_{\underline{p}}^2\rightarrow I_{\underline{p}}\rightarrow 0$$ to be two free resolutions of $I_{\underline{p}}$.

Obviously, the first resolution is not minimal, so after some columns operations, we get that the matrix $\phi$ has an entry which is an invertible element in $R_{\underline{p}}$. So some polynomial combination of $A_1,A_2,B_1,B_2,C_1,C_2$ does not vanish at $[a,b,c]$. Contradiction with the assumption. So $Z=\emptyset$.

\vskip .1in

\noindent\textit{Step 2.} For every $f\in R_{d_2-d_1}$, construct the ideal $$I(f)=\langle A_2-fA_1,B_2-fB_1,C_2-fC_1 \rangle.$$ If we show that there exists an $f$ such that $V(I(f))=\emptyset$ we will be done (see the result we stated before the theorem).

Suppose that for all $f\in R_{d_2-d_1}$, $V(I(f))\neq \emptyset$. For each $f$, let $$P_f\in V(I(f)).$$ Since $A_2(P_f)-f(P_f)A_1(P_f)=0,B_2(P_f)-f(P_f)B_1(P_f)=0,C_2(P_f)-f(P_f)C_1(P_f)=0$, rewriting the generators of $I$ in a convenient way (for example, $f_1=B_1(C_2-fC_1)-C_1(B_2-fB_1)$), we obtain that $P_f\in V(I)$. So for each $f\in R_{d_2-d_1}$, we have $V(I(f))\subset V(I)$.

\vskip .1in

\noindent\textit{Step 3.} If $[a,b,c]\in V(I(f))\cap V(I(g))$ for some $f,g\in R_{d_2-d_1}$, then $f(a,b,c)=g(a,b,c)$.

At Step 1, we saw that $Z=\emptyset$. So one of the $A_1,A_2,B_1,B_2,C_1,C_2$ does not vanish at $[a,b,c]$. If for example $A_2(a,b,c)\neq 0$, since $A_2(a,b,c)-f(a,b,c)A_1(a,b,c)=0$, we get that $A_1(a,b,c)\neq 0$.

Without loss of generality we may assume that $A_1(a,b,c)\neq 0$. Since $[a,b,c]\in V(I(f))\cap V(I(g))$, we get that $$f(a,b,c)=g(a,b,c)=\frac{A_2(a,b,c)}{A_1(a,b,c)}.$$

\vskip .1in

\noindent\textit{Step 4.} Let $f\in R_{d_2-d_1}$, generic enough (i.e. $f$ does not vanish at any point of $V(I)$). Consider the sequence of polynomials $$f,2\cdot f,3\cdot f,\ldots,k\cdot f,\ldots.$$ Then for every $i\neq j$, $$V(I(i\cdot f))\cap V(I(j\cdot f))=\emptyset.$$ Otherwise, if $[a,b,c]$ is a point in the intersection, from Step 3 we should get that $i\cdot f(a,b,c)=j\cdot f(a,b,c)$, and so $f(a,b,c)=0$. But this will contradict the generic condition on $f$, since by Step 2, $$[a,b,c]\in V(I(i\cdot f))\subset V(I),$$ and so $f$ would vanish at a point of $V(I)$.

So for each $k\geq 1$ we obtain a distinct point in $V(I(k\cdot f))\subset V(I)$. So the cardinality of $V(I)$ is infinity. Contradiction, since $ht(I)=2$.

In conclusion, for this particular $f$, $V(I(f))=\emptyset$.
\end{proof}

For the first implication, the condition that $I$ is locally a complete intersection is not used. In fact, as the referee observed, if $I$ has a regular syzygy, then $I$ is locally a complete intersection, and Theorem \ref{thm:main} can be restated: let $I\subset R=\mathbb C[x,y,z]$ be an ideal of height 2, minimally generated by three homogeneous polynomials of the same degree. Then, $I$ is locally a complete intersection and $R/I$ is arithmetically Cohen-Macaulay $\Leftrightarrow$ $I$ has a regular syzygy.

\vskip .2in

Now we apply this result to divisors in $\mathbb P^2$. Let $Y=V(F)$ be a reduced (not necessarily irreducible) curve in $X=\mathbb P^2$, then $Y$ is {\em free} if and only if the $R-$module $$D(Y):=\{\theta\in Der_{\mathbb C}(R)|\theta(F)\in \langle F\rangle R\},$$ is a free $R-$module, where $R=\mathbb C[x,y,z]$.

Since $F\in R$ is a homogeneous polynomial, one has the Euler relation $\deg(F)\cdot F=xF_x+yF_y+zF_z,$ and we have a splitting $$D(Y)=\theta_E\cdot R\oplus D_0(Y),$$ where $D_0(Y)=D(Y)/\theta_E\cdot R$ is isomorphic to the first syzygy module of the Jacobian ideal $J_F=\langle F_x,F_y,F_z\rangle$ of $F$ and $\theta_E=x\frac{\partial}{\partial x}+y\frac{\partial}{\partial y}+z\frac{\partial}{\partial z}$ is the Euler derivation. Indeed, if $\theta\in D(Y)$, then $\theta(F)=DF,$ for some $D\in R$. But we can write $DF=(\frac{D}{\deg(F)}\cdot\theta_E)(F).$ So modulo $\theta_E\cdot R$, the coefficients of $\theta\in D(Y)$ in the standard basis of $Der(R)$ can be viewed as first syzygy on $F_x,F_y,F_z$. The splitting of $D(Y)$ comes from the fact that if $AF_x+BF_y+CF_z=0$ is a syzygy, then $\theta =A\frac{\partial}{\partial x}+B\frac{\partial}{\partial y}+C\frac{\partial}{\partial z}$ is an element in $D(Y)$, as $\theta(F)=0$.

Suppose $Y$ has isolated singularities (i.e., $ht(J_F)=2$). Then, from above, $D(Y)$ is free $R-$module if and only if the first syzygies module of $J_F$ is free (of rank $2$). This leads to the following well-known criterion (Saito's Criterion\footnote{ Originally (\cite{s}) this criterion was stated as follows: if $\theta_i=a_i\partial_x+b_i\partial_y+c_i\partial_z,i=1,2,3$ are in $D(Y)$ such that the determinant of the $3\times 3$ matrix with columns $(a_i,b_i,c_i)$ is a nonzero scalar multiple of $F$, then $Y$ is free. Replacing one of the $\theta_i$ with $\theta_E$, it comes to show that $F_x,F_y,F_z$ are the maximal minors of a $2\times 3$ matrix, which by Hilbert-Burch Theorem is equivalent to $R/J_F$ being arithmetically Cohen-Macaulay. One can also check \cite{sim1}, Proposition 3.7.}) for freeness: $$Y\mbox{ is free if and only if }R/J_F\mbox{ is arithmetically Cohen-Macaulay}.$$

Let $V(f)\subset \mathbb C^2$ be a reduced curve with isolated singularity at $(0,0)\in V(f)$. $(0,0)$ is called {\em quasihomogeneous} singularity if $f(x,y)=\sum c_{i,j} x^iy^j$ is weighted homogeneous (i.e., there exists rational numbers $\alpha,\beta$ such that $f(x^{\alpha},y^{\beta})$ is homogeneous). $V(J_F)$ describes the singular locus of $Y=V(F)$, and we say that $Y$ is {\em quasihomogeneous} if every $P\in V(J_F)$ is a quasihomogeneous singularity for $F_P$ (the localization of $F$ at $P$).

Every arrangement of lines in $\mathbb P^2$ is quasihomogeneous (see Remark \ref{rem:lci} below).

Now we are ready to prove the following criterion for freeness of a special class of divisors on $\mathbb P^2$.

\begin{thm}\label{thm:freediv} Let $Y=V(F)$ be a quasihomogeneous reduced curve in $\mathbb P^2$. $Y$ is free if and only if there exists a syzygy $AF_x+BF_y+CF_z=0$ with $\{A,B,C\}$ forming an $R-$regular sequence.
\end{thm}
\begin{proof} The result is immediate from Theorem \ref{thm:main}, if we show that $J_F$ is locally a complete intersection. Also, the side note at the beginning of Proof 1 is useful to find a regular syzygy on the specific set of minimal generators of $J_F$, consisting of the partial derivatives of $F$.

After a change of coordinates we may assume that $Y$ has no singularities on the line of equation $z=0$. Let $f(x,y)=F(x,y,1)$. Then the dehomogenization $z\mapsto 1$ gives an isomorphism of rings $$R/J_F\rightarrow \mathbb C[x,y]/\langle f_x,f_y,f\rangle.$$

It is enough to show that the ideal $\langle f_x,f_y,f\rangle$ is locally a complete intersection. Let $P\in \mathbb P^2$ be a singularity of $Y$. Let's assume that $P=[0,0,1]$. Then $P=(0,0)$ is an isolated singularity of $V(f)$. Since $P$ is quasihomogeneous, then, by \cite{st}, Section 1.3 (in fact by Reiffen, \cite{r}), $f\in\langle f_x,f_y\rangle$ in the localization at the ideal $\langle x,y\rangle$ of $P$.

Therefore $\langle f_x,f_y,f\rangle$ is locally a complete intersection. \end{proof}

\begin{rem}\label{rem:lci} The argument used by Schenck (see \cite{s2}, Lemma 2.5.) to show that the Jacobian ideal of a line arrangement (more generally, of a hyperplane arrangement) is locally a complete intersection, does not work at a first glance. Suppose $P=[0,0,1]$ is a singularity of $Y$, when $Y=V(F)\subset\mathbb P^2$ is a line arrangement. Suppose $F=G\cdot H$ with $G(P)=0$ and $H(P)\neq 0$ (i.e., $G$ is the product of the equations of the lines passing through $P$). We have
\begin{eqnarray} F_x&=&G_x\cdot H+G\cdot H_x\nonumber\\
F_y&=&G_y\cdot H+G\cdot H_y\nonumber\\
F_z&=&G_z\cdot H+G\cdot H_z.\nonumber
\end{eqnarray} In the localization at the ideal of $P$, $H$ and $H_z$ are invertible. Furthermore we have $G\in \mathbb C[x,y]$, and therefore $G_z=0$. Since from Euler relation one has $\deg(G)\cdot G=xG_x+yG_y$, we get that locally $F_z\in\langle F_x,F_y\rangle$.

If $G$ is not a product of linear forms, $G_z$ does not vanish and the proof may not work as nicely as for the case of line arrangements.
\end{rem}

It would be interesting to have a similar result for height 2 ideals in $\mathbb C[x_1,\ldots,x_k]$, locally complete intersection minimally generated by $k$ homogeneous polynomials. In our situation, when $k=3$, the first part of the proof uses the fact that if we have a regular syzygy then the minimal generators are the maximal minors of a $2\times 3$ matrix and we used Hilbert-Burch theorem. As we'll see below this is not the case when $k\geq 4$, and a ``natural'' generalization to higher dimensions will not work. Also when $k=3$, the second part of the proof relies on the fact that $V(I)$ is a finite set of points in $\mathbb P^2$.

\begin{exm}\label{exm:P3} Let $Q=xyzw(x+y+z+w)\in R=\mathbb C[x,y,z,w]$, be the defining polynomial of a hyperplane arrangement in $\mathbb P^3$. Let $J_Q=\langle Q_x,Q_y,Q_z,Q_w\rangle$ be the Jacobian ideal of $Q$. We have $ht(J_Q)=2$, $J_Q$ is locally a complete intersection (see Remark \ref{rem:lci}), but $R/J_Q$ is not arithmetically Cohen-Macaulay (i.e, the arrangement is not free).

We have that $J_Q$ is minimally generated by $g_1,g_2,g_3,g_4\in J_Q$, where \begin{eqnarray}
g_1&=&x^2yz+xy^2z+xyz^2+2xyzw\nonumber\\
g_2&=&x^2yw+xy^2w+2xyzw+xyw^2\nonumber\\
g_3&=&x^2zw+2xyzw+xz^2w+xzw^2\nonumber\\
g_4&=&xyzw+\frac{1}{2}y^2zw+\frac{1}{2}yz^2w+\frac{1}{2}yzw^2.\nonumber
\end{eqnarray}

We also have $$s_1g_1+s_2g_2+s_3g_3+s_4g_4=0,$$ where \begin{eqnarray}
s_1&=&-8xw-4zw-3w^2\nonumber\\
s_2&=&4xz+yz+2z^2+3zw\nonumber\\
s_3&=&-y^2-4yw\nonumber\\
s_4&=&4x^2+2xy+14xw.\nonumber
\end{eqnarray}

Noting that $ht(\langle s_1,s_2,s_3,s_4\rangle)=4$, gives that $\{s_1,s_2,s_3,s_4\}$ is an $R-$regular sequence.
\end{exm}

\begin{exm} In this example we can see that the condition of quasihomogeneity is essential for Theorem \ref{thm:freediv} to work. Let $Y=V(x(x+y)(x-y)(x+2y)(x^2+yz))$ be a union of four lines and an irreducible conic in $\mathbb P^2$, all passing through the point $P=[0,0,1]$.

$Y$ is not quasihomogeneous at $P$: let $$f(x,y)=F(x,y,1)=x^6+2x^5y-x^4y^2-2x^3y^3+x^4y+2x^3y^2-x^2y^3-2xy^4.$$ If $x$ has weight $\alpha$ and $y$ has weight $\beta$, then we should have that $6\alpha=5\alpha+\beta=4\alpha+\beta,$ and therefore $\alpha=\beta=0$. We can see this also by following the arguments in \cite{st}, Section 1.3. We have $\deg(J_F)=19$ which equals the sum of Tjurina numbers at each singularity. A curve is quasihomogeneous at a singularity if Tjurina and Milnor numbers are equal. In our case, our singularities are transverse (i.e., the components of $Y$ are not tangent), and therefore the Milnor number at each singularity $P$ equals to $(n_P-1)^2$, where $n_p$ is the number of distinct branches of $Y$ at $P$. The sum of the Milnor numbers equals $20$, so at a singularity (in our case $P$) these two numbers differ.

With Macaulay 2 we get the graded minimal free resolution: $$0\rightarrow R(-7)\oplus R(-8)\stackrel{\phi}\rightarrow R^3(-5)\rightarrow R\rightarrow R/J_F\rightarrow 0.$$ Therefore $Y$ is free. One should observe (also from Macaulay 2) that each entry in the syzygies matrix $\phi$ belongs to the ideal $\langle x,y\rangle$ of the point $P$. So if we localize at this ideal the resolution remains minimal (so at $P$, $J_F$ is not locally a complete intersection) and the Step 1 in the proof of Theorem \ref{thm:main} does not work.

This step in the proof can not be avoided, since any syzygy $(A,B,C)$ has $\langle A,B,C\rangle\subseteq\langle x,y\rangle$, and therefore it cannot form a regular sequence.
\end{exm}

\begin{exm} Next we give an example of a quasihomogeneous divisor that is free. Let $Y=V(y(x^2+yz))$ be a union of a conic and a line tangent to the conic at $P=[0,0,1]$. So the singularity $P$ is not transverse. We have $f(x,y)=F(x,y,1)=x^2y+y^2$ and this is weighted homogeneous with the weight of $x$ being $\alpha=1$, and the weight of $y$ being $\beta=2$. So $Y$ is quasihomogeneous.

$J_F=\langle xy,x^2+2yz,y^2\rangle$. Observe that $$(-x)xy+y(x^2+2yz)+(-2z)y^2=0$$ and $\{-x,y,-2z\}$ is an $R-$regular sequence. So $Y$ is free.
\end{exm}

\section{Applications and connections to the reduced Jacobian scheme}

Let $Y=V(F)$ be a reduced (not necessarily irreducible) curve in $X=\mathbb P^2$. If $AF_x+BF_y+CF_z=0$ is a syzygy, then $\theta=A\frac{\partial}{\partial x}+B\frac{\partial}{\partial y}+C\frac{\partial}{\partial z}$ belongs to $D(Y)$, and it will be called a \textit{special logarithmic derivation}.

The recipe to obtain a special logarithmic derivation from a non-special logarithmic derivation that is not a multiple of $\theta_E$ is immediate: if $\theta_1=A_1\frac{\partial}{\partial x}+B_1\frac{\partial}{\partial y}+C_1\frac{\partial}{\partial z}\in D(Y)$ is a logarithmic derivation, not a multiple of the Euler derivation $\theta_E$, then $\theta(F)=D\cdot F$ for some $D\in R$. After using Euler relation on $F$ in the right-hand side we obtain: $$(A_1-xD/n)F_x+(B_1-yD/n)F_y+(C_1-zD/n)F_z=0,$$ where $n=\deg(F)$.

If $\{A,B,C\}$ above is a regular sequence and if $\deg(A)=\deg(B)=\deg(C)=d$, $\theta$ will be called a \textit{regular special logarithmic derivation of degree $d$}. From Theorem \ref{thm:freediv}, every free quasihomogeneous reduced curve $Y=V(F)$ in $\mathbb P^2$ has a regular special logarithmic derivation.

When $D(Y)$ is a free $R-$module (of rank 3), the degrees of the elements in a basis for $D(Y)$ are called {\em the exponents of $Y$}, denoted with $exp(Y)=\{1,a,b\}$ (the number 1 comes from the degree of the Euler derivation). If $Y$ is free, then it should have at least one regular special logarithmic derivation of degree $a$ or $b$. If $a\leq b$, and if the syzygy of degree $a$ is not regular, then there exists a syzygy of degree $b$ which forms a regular sequence. The way one should construct this syzygy is explained in Step 4 in the proof of Theorem \ref{thm:main}.

\vskip .1in

{\em The reduced Jacobian scheme of $Y=V(F)\subset\mathbb P^2$} is the reduced zero-dimensional scheme of the singular locus of $Y$. It has as defining ideal $\sqrt{J_F}$.

\subsection{Line arrangements in $\mathbb P^2$.} Let $\mathcal A=V(Q)$ be a central essential arrangement in $\mathbb C^3$ (i.e. a line arrangement in $\mathbb P^2$). As a consequence to Theorem \ref{thm:freediv}, we give an upper bound for $\deg(\sqrt{J_Q})$, when $\mathcal A$ is free. A priori the degree of this scheme is less or equal than ${{|\mathcal A|}\choose{2}}$, and at this moment we are not aware of any other bound. Very little is known about the reduced Jacobian scheme of an arrangement, $V(J_Q)$, and about the connections between this scheme and the arrangement. In fact, Wakefield and Yoshinaga give a very nice and simple example (\cite{wy}, Example 4.2) of two rank 3 central arrangements completely different (they have different number of hyperplanes, and one is supersolvable and the other is not) with the same $\sqrt{J_Q}$, emphasizing that the reduced Jacobian scheme has few information about the arrangement itself.

\begin{prop} \label{prop:application}Let $\mathcal A$ be a free line arrangement in $\mathbb P^2$ with a regular special logarithmic derivation of degree $d\geq 2$. Then $$\deg(\sqrt{J_Q})\leq d^2+d+1.$$
\end{prop}
\begin{proof} Let $\theta=A\frac{\partial}{\partial x}+B\frac{\partial}{\partial y}+C\frac{\partial}{\partial z}$ be the regular special logarithmic derivation of degree $d\geq 2$. This means that $AQ_x+BQ_y+CQ_z=0$ and $\{A,B,C\}$ is a regular sequence.

Suppose $Q=L_1\cdots L_n,$ where we fixed the defining linear forms: $L_i=a_ix+b_iy+c_iz$. Therefore $$(*)\mbox{ }a_iA+b_iB+c_iC=L_iS_i, \mbox{ for some }S_i\in R,\mbox{ for all }i=1,\ldots,n.$$

Let $[a,b,c]\in V(J_Q)$. Then $[a,b,c]$ is the intersection of at least two lines in $\mathcal A$. So there exist $i\neq j$ such that $L_i(a,b,c)=0$ and $L_j(a,b,c)=0$.

Two things can happen: either $[a,b,c]\in V(A,B,C)$, or not.

If not, then the point $[A(a,b,c),B(a,b,c),C(a,b,c)]\in\mathbb P^2$ is also the intersection of the lines of equations $L_i=0$ and $L_j=0$ (evaluate $(*)$ at the point $[a,b,c]$). This means that there exists a nonzero constant $k$ such that $$A(a,b,c)=ka,B(a,b,c)=kb,C(a,b,c)=kc.$$ Denoting with $$I^{(A,B,C)}=\langle yA-xB,zA-xC,zB-yC\rangle,$$ we see that $$[a,b,c]\in V(I^{(A,B,C)}).$$

If $[a,b,c]\in V(A,B,C)$, then, since $I^{(A,B,C)}\subset \langle A,B,C\rangle$, we have that $V(A,B,C)\subseteq V(I^{(A,B,C)})$.

We obtained that $$V(J_Q)\subseteq V(I^{(A,B,C)}).$$

\vskip .1in
\noindent \textit{Claim 1:} $I^{(A,B,C)}$ is minimally generated by three polynomials.

Suppose there exist constants $\alpha,\beta\in\mathbb C$ such that $$yA-xB=\alpha(zA-xC)+\beta(zB-yC).$$ Then we obtain a linear syzygy on $A,B,C$: $$(y-\alpha z)A+(-x-\beta z)B+(\alpha x+\beta y)C=0.$$

If $d\geq 2$, then this is impossible; one cannot have a linear syzygy on a regular sequence consisting of homogeneous polynomials all of degree $\geq 2$.

\vskip .1in
\noindent \textit{Claim 2:} $ht(I^{(A,B,C)})=2$.

Since $I^{(A,B,C)}\subset\sqrt{J_Q}$, and $ht(J_Q)=2$, then $ht(I^{(A,B,C)})\leq 2$. If $ht(I^{(A,B,C)})=1$, since $R=\mathbb C[x,y,z]$ is a unique factorization domain, there exists $D\in R$, non-constant homogeneous polynomial, dividing each of the generators of $I^{(A,B,C)}$.

So we have
\begin{eqnarray}
yA-xB&=&fD\nonumber\\
zA-xC&=&gD\nonumber\\
zB-yC&=&hD,\nonumber
\end{eqnarray} with $ht(f,g,h)=2$.

We obtain that $$zf-yg+xh=0.$$ Since $\{x,y,z\}$ form a regular sequence we get that $$f=yv_1+xv_2,g=zv_1+xv_3, h=yv_3-zv_2,$$ where $v_i\in R$.

Plugging these back into the equations above we get that
\begin{eqnarray}
y(A-v_1D)-x(B+v_2D)&=&0\nonumber\\
z(A-v_1D)-x(C+v_3D)&=&0\nonumber\\
z(B+v_2D)-y(C+v_3D)&=&0\nonumber
\end{eqnarray}

But these mean that there exists $w\in R$ such that
\begin{eqnarray}
A-v_1D&=&xw\nonumber\\
B+v_2D&=&yw\nonumber\\
C+v_3D&=&zw\nonumber
\end{eqnarray}

And therefore, $\langle A,B,C\rangle\subset\langle D,w\rangle$. Contradiction.

\vskip .1in

We just showed that $ht(I^{(A,B,C)})=2$ and by the way it is defined, $R/I^{(A,B,C)}$ has the Hilbert-Burch minimal free resolution: \footnote{Here we could use also Theorem \ref{thm:Thm0}, by observing that $z(yA-xB)-y(zA-xC)+x(zB-yC)=0$, and $\{z,-y,x\}$ is an $R-$regular sequence.}
$$0\rightarrow R(-(d+2))\oplus R(-(2d+1))\rightarrow R^3(-(d+1))\rightarrow R\rightarrow R/I^{(A,B,C)}\rightarrow 0.$$

Once we have this, since $ht(I^{(A,B,C)})=2$, the Hilbert polynomial computation gives the degree of $I^{(A,B,C)}$: $$\deg(I^{(A,B,C)})=d^2+d+1.$$

Since $I^{(A,B,C)}\subseteq\sqrt{J_Q}$, finally we obtain the desired upper bound.\end{proof}

\begin{exm}\label{exm:example1} The condition $d\geq 2$ in Proposition \ref{prop:application} is essential. Consider the following free (supersolvable) arrangement: a pencil of $n-1$ lines through the point $[0,0,1]$ and the line at infinity of equation $z=0$. The defining polynomial of this arrangement is $$Q=zP,$$ where $P$ is the defining polynomial of the pencil of lines, so $P\in\mathbb C[x,y]$.

Since $Q_x=zP_x,Q_y=zP_y,Q_z=P$, and since $(n-1)P=xP_x+yP_y,$ we get that $\theta=x\frac{\partial}{\partial x}+y\frac{\partial}{\partial y}-(n-1)z\frac{\partial}{\partial z}$ is a regular special logarithmic derivation of degree 1.

Of course, $|V(J_Q)|=n$ which is not less or equal than $3=1^2+1+1$, if $n\geq 4$. The reason why the proof of Proposition \ref{prop:application} does not work for this example is that $I^{(x,y,-(n-1)z)}=\langle xz,yz\rangle$ has codimension 1.
\end{exm}

\begin{rem}\label{rem:remark1} The situation described in Example \ref{exm:example1} is unique, in the sense that if a line arrangement $\mathcal A\subset \mathbb P^2, |\mathcal A|=n,$ has a special logarithmic derivation of degree 1 (not necessarily regular), then it must consist of a pencil of $n-1$ lines through a point, and another line missing that point.

Let $Q=\prod_{i=1}^nL_i,$ with $L_i=a_ix+b_iy+c_iz$, and let $M_i=\alpha_ix+\beta_iy+\gamma_iz, i=1,2,3$ be three linear forms such that $$M_1Q_x+M_2Q_y+M_3Q_z=0.$$ Since $\theta=M_1\frac{\partial}{\partial x}+M_2\frac{\partial}{\partial y}+M_3\frac{\partial}{\partial z}$ is a logarithmic derivation (we have $\theta(Q)=0\in\langle Q\rangle$; so $\theta$ is a special logarithmic derivation of degree 1), from the product rule of derivations: $\theta(Q)=L_1\theta(L_2\cdots L_n)+\theta(L_1)L_2\cdots L_n$, and from the syzygy equation above we have $$M_1a_i+M_2b_i+M_3c_i=L_id_i, \mbox{ for all }i=1,\ldots,n,$$  where $d_i\in\mathbb C\mbox{ with }\sum_{i=1}^nd_i=0$.

Without any loss of generality, we may assume that $L_1=x,L_2=y,L_3=z$. Then we obtain:
\begin{eqnarray}
\alpha_1=d_1,\beta_1=0,\gamma_1=0\nonumber\\
\alpha_2=0,\beta_2=d_2,\gamma_2=0\nonumber\\
\alpha_3=0,\beta_3=0,\gamma_3=d_3.\nonumber
\end{eqnarray}

From these, for $i\geq 4$ we obtain that $$(\sharp) \mbox{ }(d_1-d_i)a_i=0,(d_2-d_i)b_i=0,(d_3-d_i)c_i=0.$$

If $d_i\neq d_1$ (or $d_i\neq d_2$, or $d_i\neq d_3$) for all $i\geq 4$, then from $(\sharp)$ we have $a_i=0$ (or $b_i=0$, or, respectively, $c_i=0$) for all $i\geq 4$. But this will mean that the lines $V(L_2),V(L_3),V(L_4),\ldots,V(L_n)$ will all pass through the point $[1,0,0]$. And the line $V(L_1)$ does not pass through this point. We have analogous results in the other two situations.

Let us assume that $d_4=\cdots=d_m=d_1$, and $d_j\neq d_1$, for all $n\geq j\geq m+1$.

Let $j_0$ be such that $n\geq j_0\geq m+1$. If $d_{j_0}\neq d_2$ (or $d_{j_0}\neq d_3$), then, since $d_{j_0}\neq d_1$, we obtain from $(\sharp)$ that $a_{j_0}=0$ and $b_{j_0}=0$ (or, respectively, $c_{j_0}=0$). This means that $V(L_{j_0})=V(L_3)$ (or $V(L_{j_0})=V(L_2)$). Contradiction.

We obtained that $d_{m+1}=\cdots=d_n=d_2=d_3=d_{j_0}\neq d_1$. We also have $d_1=d_4\neq d_2$ and $d_1=d_4\neq d_3$. So from $(\sharp)$, we get $b_4=c_4=0$ and so $V(L_4)=V(L_1)$. Contradiction. So the existence of the index $j_0$ is impossible. Hence $$d_4=\cdots=d_n=d_1.$$

If $d_4\neq d_2$ and $d_4\neq d_3$ we obtain again that $b_4=c_4=0$ and so $V(L_4)=V(L_1)$. Contradiction. If $d_4=d_2=d_3$, we have $d_1=d_2=d_3=d_4=\cdots=d_n$, leading to a contradiction with the fact that $\sum_{i=1}^nd_i=0$.

So we can suppose that $d_1=d_3=d_4=\cdots=d_n\neq d_2$. Hence we are in the situation $d_i\neq d_2$, for all $i\geq 4$, described in the paragraph after we listed equations $(\sharp)$. This means that $b_i=0$, for all $i\geq 4$, and so $V(L_1),V(L_3),V(L_4),\ldots,V(L_n)$ all pass through the point $[0,1,0]$. The remaining line $V(L_2)=V(y)$ misses this point.

So we must have that $d_i\neq d_1$ (or $d_i\neq d_2$, or $d_i\neq d_3$) for all $i\geq 4$.
\end{rem}

\begin{exm}\label{exm:example2} Consider $\mathcal A\subset\mathbb P^2$ with defining polynomial $$Q=xyz(x-y)(x-z)(y-z).$$ This is a free arrangement (supersolvable) with exponents $exp(\mathcal A)=\{1,2,3\}$.

This means that the first syzygies module is free with a basis consisting of two syzygies of degrees 2 and 3, respectively. One can check that $$(-x^2+xy+xz)Q_x+(xy-y^2+yz)Q_y+(xz+yz-z^2)Q_z=0$$ gives the syzygy of degree 2. Since $ht(\langle-x^2+xy+xz,xy-y^2+yz,xz+yz-z^2\rangle)=2$, this syzygy is not regular.

It remains that the regular syzygy must have degree 3, and therefore $d=3$ in Proposition \ref{prop:application}.

We also have that $\deg(\sqrt{J_Q})=7$, and indeed we have $7\leq 3^2+3+1=13$.

Is it possible to construct examples when the bound in Proposition \ref{prop:application} is attained? Can the upper bound be improved?
\end{exm}

Proposition \ref{prop:application} gives a criterion for non-freeness.

\begin{cor}\label{cor:corollary2} Let $\mathcal A\subset\mathbb P^2$ be an essential arrangement of $n$ lines, that is not a union of a pencil of $n-1$ lines and a generic line (i.e. the situation in Remark \ref{rem:remark1}). If $\deg(\sqrt{J_Q})\geq n^2-5n+8$, then $\mathcal A$ is not free.
\end{cor}
\begin{proof} Suppose $\mathcal A$ is free with exponents $exp(\mathcal A)=\{1,a,b\}$. Suppose $a\leq b$. Since $\mathcal A$ is not of the type described in Remark \ref{rem:remark1}, then $a\geq 2$. Therefore, since $a+b=n-1$, we have $b\leq n-3$.

From Proposition \ref{prop:application}, we have $$\deg(\sqrt{J_Q})\leq b^2+b+1\leq (n-3)^2+(n-3)+1=n^2-5n+7.$$ Contradiction.
\end{proof}

\subsection{Other results concerning free divisors in $\mathbb P^2$}

In what follows we will see how the degrees of the first syzygies on the Jacobian ideal relate to the degrees of the generators of the radical of the Jacobian ideal.

\begin{prop} \label{prop:proposition2} Let $Y=V(F)\subset\mathbb P^2$ be an arrangement of $m$ smooth irreducible curves $V(F_i)$, $m\geq 2$, that intersect transversely. Let $\beta(J_F)$ be the minimal degree of a syzygy on the Jacobian ideal $J_F$ of $F$. Let $\alpha(\sqrt{J_F})$ be the minimum degree of a generator of $\sqrt{J_F}$. Then $$\alpha(\sqrt{J_F})\leq \beta(J_F)+1.$$
\end{prop}
\begin{proof} The proof is similar to the beginning of the proof of Proposition \ref{prop:application}. Let $AF_x+BF_y+CF_z=0$ be a syzygy of degree $\beta(J_F)$. Consider $\theta=A\frac{\partial}{\partial x}+B\frac{\partial}{\partial y}+C\frac{\partial}{\partial z}$ in $D(Y)$.

Since $F=F_1\cdots F_m$, from properties of derivations we have that $$(**)\mbox{ }A(F_i)_x+B(F_i)_y+C(F_i)_z=F_i\cdot g_i,\mbox{ with }g_i\in R, i=1,\ldots,m.$$

\vskip .1in

Let $P=[a,b,c]\in V(J_F)$. By the way we defined $Y$, one has $P\in V(F_i)\cap V(F_j)$, for some indices $i\neq j$.

If $P\notin V(A,B,C)$, by evaluating $(**)$ at $P$, we obtain that the point $[A(P),B(P),C(P)]$ is on the tangent line $L_i$ to $V(F_i)$ at $P$, as well as on the tangent line $L_j$ to $V(F_j)$ also at $P$. Since $V(F_i)$ and $V(F_j)$ intersect transversely, the lines $L_i$ and $L_j$ have different slopes and therefore $$P=[A(P),B(P),C(P)].$$

Denoting with $$I^{(A,B,C)}=\langle yA-xB,zA-xC,zB-yC\rangle,$$ we see that $$[a,b,c]\in V(I^{(A,B,C)}).$$

If $P\in V(A,B,C)$, then, since $I^{(A,B,C)}\subset \langle A,B,C\rangle$, we have that $V(A,B,C)\subseteq V(I^{(A,B,C)})$.

We obtained that $$V(J_F)\subseteq V(I^{(A,B,C)}),$$ and hence the assertion.
\end{proof}

The bound in Proposition \ref{prop:proposition2} can be attained as we can see in Example \ref{exm:example2} when we have $\beta(J_Q)=2$ and $\alpha(\sqrt{J_Q})=3$.

\vskip .1in

We end with an immediate corollary to the proof of Proposition \ref{prop:proposition2}.

\begin{cor} Let $Y=V(F)\subset\mathbb P^2$ be an arrangement of $m$ smooth irreducible curves $V(F_i)$, $m\geq 2$, that intersect transversely. Suppose that $Y$ is free. Then, if $(A_1,B_1,C_1)$ and $(A_2,B_2,C_2)$ are a basis for the first syzygy module of $J_F$, we have $$V(J_F)= V(I^{(A_1,B_1,C_1)})\cap V(I^{(A_2,B_2,C_2)}).$$
\end{cor}

\begin{exm} In Example \ref{exm:example1} we have $(A_1,B_1,C_1)=(x,y,-(n-1)z)$, and $(A_2,B_2,C_2)=(P_y,-P_x,0)$. We have $I_1=I^{(x,y,-(n-1)z)}=\langle xz,yz\rangle$ and $I_2=I^{(P_y,-P_x,0)}=\langle P,zP_y,zP_x\rangle$.

Let $q\in V(I_1)\cap V(I_2)$. Then, $q$ is either $q=[a,b,0]$ with $P(q)=0$, and so $q$ is one of the intersection points between the pencil and the line at infinity, or the other possibility for $q$ is to be $q=[0,0,1]$ which is the multiple point of the pencil.
\end{exm}

\vskip .1in

\noindent\textbf{Acknowledgments}: I would like to express my sincere gratitude to Graham Denham for the inspiring discussions we had on the subject. I thank Max Wakefield for suggesting Example \ref{exm:example2} and Corollary \ref{cor:corollary2}, and all the participants at Terao Fest 2011 for the feedback on generalizing Theorem \ref{thm:freediv} to arbitrary rank hyperplane arrangements. Their questions led to Example \ref{exm:P3}.

I thank the anonymous referee for bringing to my attention the references \cite{eh} and \cite{k}, and for all the useful suggestions and comments.

All the computations were done using Macaulay 2 by D. Grayson and M. Stillman.

\renewcommand{\baselinestretch}{1.0}
\small\normalsize 

\bibliographystyle{amsalpha}

\begin{thebibliography}{10}

\bibitem{cs} Cox, D., Schenck, H. (2003).
            Local complete intersections in $\mathbb P^2$ and Koszul syzygies.
            {\em Proceedings of the A.M.S.} 131: 2007-2014.

\bibitem{dssww} Denham, G., Schenck, H., Schulze, M., Wakefield, M., Walther, U. (2011).
            Local cohomology of logarithmic forms.
            {\em arXiv: 1103.2459}, to appear in {\em Ann. Inst. Fourier}.

\bibitem{e} Eisenbud, D. (1995).
            Commutative Algebra with a View Toward Algebraic Geometry.
            Springer-Verlag, New York.

\bibitem{eh} Eisenbud, D., Huneke, C. (1990).
            Ideals with a regular sequence as syzygy.
            Appendix to ``Sur les hypersufaces dont les sections hyperplanes sont module constant'' by Arnaud Beauville. pp. 121-133.
            Grothendieck Festschrift, Vol. 1.
            {\em Progress in Mathematics} 86: 132-133.

\bibitem{k} Kustin, A. (1986).
            The minimal free resolution of the Huneke-Ulrich deviation two Gorenstein ideals.
            {\em J. Algebra} 100: 265-304.
            
\bibitem{AA} Nejad, A., Simis, A. (2011).
        The Aluffi Algebra.
        {\em Journal of Singularities} 3: 20-47.

\bibitem{ot} Orlik, P., Terao, H. (1992).
            Arrangements of Hyperplanes.
            Springer-Verlag, Berlin-Heidelberg-New York.

\bibitem{r} Reiffen, H. (1967).
            Das Lemma von Poincar\'{e} f\"{u}r holomorphe Differentialformen auf komplexen
R\"{a}umen.
            {\em Mathematische Zeitschrift} 101: 269-284.

\bibitem{s} Saito, K. (1980).
            Theory of logarithmic differential forms and logarithmic vector fields.
            {\em J. Fac. Sci. Univ. Tokyo Sect. IA Math.} 27: 265-291.

\bibitem{s2} Schenck, H. (2000).
             A rank two vector bundle associated to a three arrangement, and its Chern polynomial.
             {\em Advances in Mathematics} 149: 214-229.

\bibitem{st} Schenck, H., Tohaneanu, S. (2009).
            Freeness of conic-line arrangements in $\mathbb P^2$.
            {\em Comentarii Mathematici Helvetici} 84: 235-258.

\bibitem{sc} Schulze, M. (2010).
            Freeness and multirestriction of hyperplane arrangements.
            {\em arXiv:1003.0917}, to appear in {\em Compositio Math}.

\bibitem{sim1} Simis, A. (2005).
Differential idealizers and algebraic free divisors.
 In ``Commutative Algebra: Geometric, Homological, Combinatorial and Computational
Aspects'', Eds: A. Corso, P. Gimenez, M. V. Pinto and S. Zarzuela, pp. 211-226. 
Lecture Notes in Pure and Applied Mathematics, Vol. 244, Chapman \& Hall/CRC.



\bibitem{sim3} Simis, A. (2005).
        The depth of the Jacobian ring of a homogeneous polynomial in three variables.
        {\em Proc. Amer. Math. Soc.} 134: 1591-1598.

\bibitem{t2} Terao, H. (1980).
            Arrangements of hyperplanes and their freeness I.
            {\em J.Fac.Science Univ.Tokyo} 27: 293–312.

\bibitem{t} Terao, H. (1981).
            Generalized exponents of a free arrangement of hyperplanes and Shepard-Todd-
Brieskorn formula.
            {\em Inventiones Mathematicae} 63: 159-179.

\bibitem{wy} Wakefield, M., Yoshinaga, M. (2008).
            The Jacobian ideal of a hyperplane arrangement.
            {\em Math. Res. Lett.} 15: 795-799.

\bibitem{y1} Yoshinaga, M. (2005).
            On the freeness of 3-arrangements,
            {\em Bull. London Math. Soc.} 37: 126-134.

\end{thebibliography}

\end{document}